\documentclass[a4paper,11pt]{article}
\usepackage{amsmath,amssymb,amsthm,mathtools}
\usepackage{comment}
\usepackage{mathrsfs}
\usepackage{braket}

\usepackage{geometry}
\usepackage{hyperref}
\hypersetup{
	unicode=true,
	colorlinks=true,
	linkcolor=black,
	citecolor=black,
	urlcolor=blue,
	pdftitle={Tensor products and units of partial Burnside rings},
	pdfauthor={Wakatake Masahiro}
}

\theoremstyle{plain}
\newtheorem{thm}{Theorem}[section]
\newtheorem{cor}[thm]{Corollary}
\newtheorem{prop}[thm]{Proposition}

\newtheorem{lem}[thm]{Lemma}

\theoremstyle{definition}
\newtheorem{dfn}[thm]{Definition}

\theoremstyle{remark}

\newcommand{\Z}{\mathbb{Z}} 
\newcommand{\B}[1]{\Omega(#1)}
\newcommand{\D}{\mathcal{D}} 
\renewcommand{\P}{\mathcal{P}} 
\newcommand{\Sub}[1]{{\rm S}({#1})}

\title{Tensor products and units of partial Burnside rings}
\author{Wakatake Masahiro}
\date{}

\begin{document}
\maketitle

\begin{abstract}
In this paper, we study tensor products of partial Burnside rings
relative to collections of subgroups of finite groups. We give a necessary and
sufficient condition for the canonical homomorphism from the direct
product of the unit groups to the unit group of the tensor product to
be surjective. We also describe the decomposition of the sign unit of a reducible finite Coxeter group.
\end{abstract}

{\small \noindent 
}

\section{Introduction}
Let $G$ be a finite group. The unit group of the Burnside ring of $G$ has been studied (see \cite{Bo07}, \cite{ma82}, \cite{Di79}, \cite{yo90b}).
Yoshida introduced the theory of generalized Burnside rings (GBRs for short) with respect to a family of subgroups of $G$ (see \cite{yo90a}). In 2015, Idei and Oda studied unit groups of generalized Burnside rings relative to collections of subgroups (see \cite{io15}). Further results were obtained in \cite{oty16} and \cite{oty17}.

This paper presents results on unit groups of partial Burnside rings relative to reducible finite Coxeter groups. The cases of irreducible Coxeter groups of types $B$, $D$,
$\mathrm{E}_6$, $\mathrm{E}_7$, and $\mathrm{E}_8$ were studied
in \cite{oty17}, while the case of reducible Coxeter groups whose
irreducible components are of type $A$ was studied in \cite{ow}.

The Burnside ring of a finite group $G$ is defined to be the
Grothendieck ring of the semiring generated by isomorphism classes
of finite (left) $G$-sets, where addition and multiplication are
given by disjoint unions and Cartesian products, respectively.
Denote by $\B{G}$ the Burnside ring of $G$.

If a family $\D$ of subgroups of $G$ contains $G$ and is closed
under conjugation and intersection, then $\D$ is called a collection
of $G$. We call the Grothendieck ring of the semiring generated by
isomorphism classes of finite (left) $G$-sets such that the stabilizer
of each element lies in $\D$ the partial Burnside ring relative to
$\D$ of $G$. We denote it by $\B{G,\D}$.

Let $W$ be a finite Coxeter group with Coxeter system $(W,S)$. A subgroup $P$ of $W$ is called a parabolic subgroup if there exist $J\subseteq S$ and $g\in W$ such that $P=g^{-1}\braket{J}g$. The set of all parabolic subgroups of $W$ is a collection of $W$ (see \cite{so76}); we denote it by $\P_W$.

The paper is organized as follows. In Section 2, we recall the basic definitions and results on GBRs and PBRs. In Section 3, we consider the relationships among direct products of PBRs, tensor products of PBRs, and PBRs of direct products. Section 4 applies these results to parabolic partial Burnside rings of reducible finite Coxeter groups.

For a unital ring $R$, write $R^\times$ for the unit group of $R$. For the PBR of a finite Coxeter group relative to the set of all parabolic subgroups, the following lemma is well known.
\vspace{5pt}

\noindent\textbf{Lemma 4.2 (\cite{so66}).} 
	Let $W$ be a finite Coxeter group with Coxeter system $(W,S)$. Then the element
	\[
	\varepsilon_W:=\sum_{J\subseteq S}(-1)^{|J|}[W/\braket{J}]
	\] 
	lies in $\B{W,\P_W}^\times$. 
\vspace{5pt}

We call $\varepsilon_W$ a \emph{sign unit} of $\B{W,\P_W}$. 

If $W$ is an irreducible finite Coxeter group of type ${\rm A}, {\rm B}, {\rm D}, {\rm E}_6, {\rm E}_7$, or ${\rm E}_8$, then the unit group of the PBR relative to $\P_W$ has cardinality $4$ (see \cite{io15}, \cite{oty17}). 

The main result concerning unit groups is a necessary and sufficient
criterion for the surjectivity of the canonical homomorphism
\[
 \prod_{i=1}^{\ell}\B{G_i,\D_i}^{\times}
 \longrightarrow
 \left(\bigotimes_{i=1}^{\ell}\B{G_i,\D_i}\right)^{\times}.
\]
For reducible finite Coxeter groups, we also prove the product formula
for the sign unit.

\section{Preliminaries}
\subsection{Notation}
Let $G$ be a finite group. Denote by $\Sub{G}$ the set of subgroups of $G$. For a family $\D$ of subgroups of $G$ that is closed under $G$-conjugation, write $\D^c$ for the set of $G$-conjugacy classes contained in $\D$. Denote by $(H)$ the $G$-conjugacy class of $H$. Denote by $[X]$ the isomorphism class of a finite $G$-set $X$. For a family $\D$ of subgroups of $G$, a $G$-set $X$ is called a $(G,\D)$-set if the stabilizer of every element of $X$ lies in $\D$. If $X$ is a finite set, write $|X|$ for the cardinality of $X$. Denote by $\bigsqcup$ the disjoint union of sets.

\subsection{Partial Burnside rings}
Let $G$ be a finite group. Then the Burnside ring $\B{G}$ of $G$ can be regarded as a free abelian group with basis $\{[G/H]\mid(H)\in\Sub{G}^c\}$. The multiplication in the ring is given by 
\[
[G/H]\cdot[G/K]=\sum_{HgK\in H\backslash G/K} [G/(H\cap gKg^{-1})].
\]
	For a family $\D$ of subgroups of $G$ that is closed under conjugation, we put 
\[
\B{G,\D}:= \braket{[G/H]\mid (H)\in\D^c}_\Z.
\]

\begin{dfn}
	Let $G$ be a finite group and let $\D$ be a family of subgroups of $G$. We call $\D$ a \emph{collection of $G$} if $\D$ satisfies the following three conditions:
	\begin{itemize}
		\item $G\in\D$
		\item $H,K\in\D\Rightarrow H\cap K\in\D$
		\item $H\in\D,g\in G \Rightarrow gHg^{-1}\in \D$
	\end{itemize}
\end{dfn}
If $\D$ is a collection of $G$, then $\B{G,\D}$ is a subring of $\B{G}$ and is called a \emph{partial Burnside ring} (PBR for short) \emph{relative to $\D$ of $G$}. 

For $K\leq G$ and a $G$-set $X$, we set
\[
{\rm inv}_K(X):=\{x\in X\mid kx=x \text{ for all } k\in K\}.
\]
If $\D$ is a collection of $G$, then for each $K\in\D$, the $\Z$-linear map $\varphi^G_K:\B{G,\D}\rightarrow\Z$ defined by $[G/H]\mapsto \#{\rm inv}_K(G/H)$ is a ring homomorphism. By \cite[Theorem 3.10]{yo90a}, the map
\[
 \varphi^G=(\varphi^G_K)_{(K)\in\D^c}
 \colon \B{G,\D}\longrightarrow \prod_{(K)\in\D^c}\Z
\]
is an injective ring homomorphism. This map is called the Burnside homomorphism. Hence the unit group of every partial Burnside ring is an elementary abelian $2$-group.

\section{PBRs of direct product groups}\label{sec:direct-product-pbrs}
Let $G_t$ be a finite group, and let $\D_t$ be a collection of $G_t$ for each $t\in\{1,\ldots,\ell\}$. Then $\prod_{i=1}^\ell \D_i$ is a collection of $\prod_{i=1}^\ell G_i$.  

\begin{lem}\label{メイン命題}
	Let $G_t$ be a finite group, and let $\D_t$ be a collection of $G_t$ for $t=1,2$. Then
	\[
	 \B{G_1,\D_1}\otimes_{\Z}\B{G_2,\D_2}\simeq \B{G_1\times G_2, \D_1\times\D_2}
	 \]
	 as rings.
\end{lem}
\begin{proof}
Put $G:=G_1\times G_2$ and $\D:=\D_1\times\D_2$.  The
correspondence
\[
 (X_1,X_2)\longmapsto X_1\times X_2
\]
is biadditive and multiplicative, and therefore induces a ring
homomorphism
\[
 \pi_2:\B{G_1,\D_1}\otimes_{\Z}\B{G_2,\D_2}
 \longrightarrow \B{G,\D}.
\]
The source is a free abelian group with basis
\[
 \bigl\{[G_1/H_1]\otimes[G_2/H_2]
       \mid (H_i)\in\D_i^c,\ i=1,2\bigr\},
\]
and the target is a free abelian group with basis
\[
 \bigl\{[G/(H_1\times H_2)]
       \mid (H_i)\in\D_i^c,\ i=1,2\bigr\}.
\]
For every pair $(H_1,H_2)$ one has
\[
 \pi_2([G_1/H_1]\otimes[G_2/H_2])
 =[(G_1\times G_2)/(H_1\times H_2)].
\]
Thus $\pi_2$ maps a basis of the source bijectively onto a basis of
the target. Hence it is a ring isomorphism.
\end{proof}
In the rest of this section, let $G_t$ be a finite group, and let $\D_t$ be a collection of $G_t$ for $t=1,\ldots,\ell$.

Lemma $\ref{メイン命題}$  shows the following.
\begin{cor}\label{メイン定理A}
	The map
	\[
	\pi_\ell : \bigotimes_{i=1}^{\ell}\B{G_i,\D_i}\longrightarrow\B{\prod_{i=1}^{\ell} G_i, \prod_{i=1}^{\ell} \D_i}
	\]
	induced by the multilinear map
	\[
	 \prod_{i=1}^{\ell}\B{G_i,\D_i}\longrightarrow
	 \B{\prod_{i=1}^{\ell}G_i,\prod_{i=1}^{\ell}\D_i}
	\]
	defined on classes of finite $G_i$-sets by
	\[
	 ([X_1],\ldots,[X_\ell])\longmapsto
	 [X_1\times\cdots\times X_\ell]
	\]
	is a ring isomorphism.
\end{cor}
Put
\[
 A_i:=\B{G_i,\D_i}\quad (i=1,\ldots,\ell),
 \qquad
 A:=\bigotimes_{i=1}^{\ell}A_i.
\]
For
\[
 \boldsymbol{H}=((H_1),\ldots,(H_\ell))
 \in\prod_{i=1}^{\ell}\D_i^c,
\]
define
\[
 \varphi_{\boldsymbol{H}}
 :=
 \varphi^{G_1}_{H_1}\otimes\cdots\otimes
 \varphi^{G_\ell}_{H_\ell}
 \colon A\longrightarrow\Z.
\]

\begin{lem}\label{tensor-Burnside-homomorphism}
The map
\[
 \Phi_A
 :=
 \prod_{\boldsymbol{H}\in\prod_{i=1}^{\ell}\D_i^c}
 \varphi_{\boldsymbol{H}}
 \colon
 A\longrightarrow
 \prod_{\boldsymbol{H}\in\prod_{i=1}^{\ell}\D_i^c}\Z
\]
is an injective ring homomorphism. Moreover, for $x\in A$, the
following conditions are equivalent:
\begin{itemize}
 \item[$(1)$] $x\in A^\times$;
 \item[$(2)$] $\varphi_{\boldsymbol{H}}(x)\in\{1,-1\}$ for every
 $\boldsymbol{H}\in\prod_{i=1}^{\ell}\D_i^c$.
\end{itemize}
In particular, $A^\times$ is a finite elementary abelian $2$-group.
\end{lem}

\begin{proof}
For each $i$, the Burnside homomorphism
\[
 \varphi^{G_i}\colon
 A_i\longrightarrow\prod_{(H_i)\in\D_i^c}\Z
\]
is injective. All groups occurring here are free abelian and hence
flat over $\Z$. Factoring the tensor product map into maps which
apply one $\varphi^{G_i}$ at a time therefore shows that
\[
 \bigotimes_{i=1}^{\ell}\varphi^{G_i}\colon
 \bigotimes_{i=1}^{\ell}A_i
 \longrightarrow
 \bigotimes_{i=1}^{\ell}
 \left(\prod_{(H_i)\in\D_i^c}\Z\right)
\]
is injective.
Under the natural identification
\[
 \bigotimes_{i=1}^{\ell}
 \left(\prod_{(H_i)\in\D_i^c}\Z\right)
 \simeq
 \prod_{\boldsymbol{H}\in\prod_{i=1}^{\ell}\D_i^c}\Z,
\]
this homomorphism is $\Phi_A$.

If $x\in A^\times$, then every $\varphi_{\boldsymbol{H}}(x)$ is a
unit of $\Z$, and hence belongs to $\{1,-1\}$. Conversely, suppose
that all $\varphi_{\boldsymbol{H}}(x)$ belong to $\{1,-1\}$. Then
\[
 \varphi_{\boldsymbol{H}}(x^2)=1
 =\varphi_{\boldsymbol{H}}(1_A)
\]
for every $\boldsymbol{H}$. The injectivity of $\Phi_A$ gives
$x^2=1_A$, so $x$ is a unit and $x^{-1}=x$.
\end{proof}

Define
\[
 \mu_\ell\colon
 \prod_{i=1}^{\ell}A_i^\times
 \longrightarrow A^\times,
 \qquad
 (u_1,\ldots,u_\ell)
 \longmapsto
 u_1\otimes\cdots\otimes u_\ell.
 \tag{3.1}
\]
This is a group homomorphism. Note that the analogous map on the
underlying rings is not additive in general.

\begin{lem}\label{kernel-unit-tensor-map}
The kernel of $\mu_\ell$ is
\[
 \ker\mu_\ell
 =
 \left\{
  (\epsilon_1 1_{A_1},\ldots,\epsilon_\ell 1_{A_\ell})
  \mathrel{}\middle|\mathrel{}
  \epsilon_i\in\{1,-1\},\quad
  \prod_{i=1}^{\ell}\epsilon_i=1
 \right\}.
\]
Consequently,
\[
 |\ker\mu_\ell|=2^{\ell-1}.
\]
\end{lem}

\begin{proof}
Let $(u_1,\ldots,u_\ell)\in\ker\mu_\ell$. By Lemma
$\ref{tensor-Burnside-homomorphism}$, for every
$\boldsymbol{H}=((H_1),\ldots,(H_\ell))$ one has
\[
 \prod_{i=1}^{\ell}\varphi^{G_i}_{H_i}(u_i)=1.
 \tag{3.2}
\]
Fix $j$, fix $H_i\in\D_i$ for $i\neq j$, and take
$H,H'\in\D_j$. Applying (3.2) with $H_j=H$ and with $H_j=H'$ gives
\[
 \varphi^{G_j}_{H}(u_j)=\varphi^{G_j}_{H'}(u_j).
\]
Thus all marks of $u_j$ have the same value
$\epsilon_j\in\{1,-1\}$. By the injectivity of the Burnside
homomorphism,
\[
 u_j=\epsilon_j1_{A_j}.
\]
Equation (3.2) then gives $\prod_{j=1}^{\ell}\epsilon_j=1$.
The converse inclusion is immediate. There are $2^{\ell-1}$ sign
vectors with product $1$.
\end{proof}

For each $j\in\{1,\ldots,\ell\}$, define the injective ring
homomorphism
\[
 f_j:\B{G_j,\D_j}\longrightarrow
 \B{\prod_{i=1}^{\ell}G_i,\prod_{i=1}^{\ell}\D_i}
\]
by
\[
 f_j([G_j/H])=
 \left[\left(\prod_{i=1}^{\ell}G_i\right)/\widehat H_j\right],
 \qquad
 \widehat H_j:=G_1\times\cdots\times G_{j-1}\times H\times
 G_{j+1}\times\cdots\times G_\ell.
\]
Equivalently, under the isomorphism $\pi_\ell$ of Corollary
$\ref{メイン定理A}$, the map $f_j$ is obtained by inserting $x$ in
the $j$-th tensor factor and the identity element in all other
factors. Its injectivity follows from its action on the standard
transitive-set basis.

\begin{lem}\label{バーンサイド準同型の分解補題}
If $x_j\in\B{G_j,\D_j}$ and $H_i\in\D_i$ for
$i=1,\ldots,\ell$, then
\[
 \varphi^G_{\prod_{i=1}^{\ell}H_i}(f_j(x_j))
 =\varphi^{G_j}_{H_j}(x_j),
\]
where $G=\prod_{i=1}^{\ell}G_i$.
\end{lem}

\begin{proof}
It suffices to take $x_j=[G_j/K]$ with $K\in\D_j$. As a $G$-set,
\[
 G/\widehat K_j\simeq
 (G_1/G_1)\times\cdots\times(G_{j-1}/G_{j-1})
 \times(G_j/K)\times
 (G_{j+1}/G_{j+1})\times\cdots\times(G_\ell/G_\ell).
\]
Therefore
\[
 \left|(G/\widehat K_j)^{\prod_{i=1}^{\ell}H_i}\right|
 =\left|(G_j/K)^{H_j}\right|
 =\varphi^{G_j}_{H_j}([G_j/K]).
\]
The assertion follows by $\Z$-linearity.
\end{proof}
We first give an elementary criterion for a sign array to be a
product of one-variable sign functions.

\begin{lem}\label{rectangle-sign-lemma}
Let $I_1,\ldots,I_\ell$ be nonempty sets and let
\[
 s\colon I_1\times\cdots\times I_\ell\longrightarrow\{1,-1\}.
\]
The following conditions are equivalent:
\begin{itemize}
 \item[$(1)$] There exist maps
 $s_i\colon I_i\longrightarrow\{1,-1\}$ such that
 \[
  s(i_1,\ldots,i_\ell)=\prod_{i=1}^{\ell}s_i(i_i).
 \]

 \item[$(2)$] For every $1\leq p<q\leq\ell$, every fixed choice of
 the remaining coordinates, and every
 $i_p,i'_p\in I_p$ and $i_q,i'_q\in I_q$, one has
 \[
 \begin{split}
 &s(\ldots,i_p,\ldots,i_q,\ldots)
  s(\ldots,i'_p,\ldots,i'_q,\ldots)\\
 &\qquad=
  s(\ldots,i_p,\ldots,i'_q,\ldots)
  s(\ldots,i'_p,\ldots,i_q,\ldots).
 \end{split}
 \tag{R}
 \]
 We refer to these relations as $(R)$.

 \item[$(3)$] Fix a base point
 $\boldsymbol{i}^0=(i_1^0,\ldots,i_\ell^0)$. Then, for every
 $(i_1,\ldots,i_\ell)$,
 \[
 \begin{split}
 s(i_1,\ldots,i_\ell)
 &=s(\boldsymbol{i}^0)^{\ell-1}
 \prod_{t=1}^{\ell}
 s(i_1^0,\ldots,i_{t-1}^0,i_t,
   i_{t+1}^0,\ldots,i_\ell^0).
 \end{split}
 \tag{3.3}
 \]
\end{itemize}
\end{lem}

\begin{proof}
Condition $(1)$ immediately implies all relations $(R)$.

Suppose that $(R)$ holds. Write
\[
 s(\boldsymbol{i})=(-1)^{c(\boldsymbol{i})},
 \qquad c(\boldsymbol{i})\in\Z/2\Z.
\]
The relations $(R)$ say that every mixed difference of $c$
vanishes. Fix $\boldsymbol{i}^0$ and put
\[
 d_t(a):=
 c(i_1^0,\ldots,i_{t-1}^0,a,i_{t+1}^0,\ldots,i_\ell^0)
 -c(\boldsymbol{i}^0).
\]
An induction on the number of coordinates in which
$\boldsymbol{i}$ differs from $\boldsymbol{i}^0$ gives
\[
 c(\boldsymbol{i})-c(\boldsymbol{i}^0)
 =\sum_{t=1}^{\ell}d_t(i_t).
\]
This is equivalent to (3.3), so $(2)$ implies $(3)$.

Finally, (3.3) gives a decomposition as in $(1)$; for example, take
\[
 \begin{split}
 s_1(i_1)&:=
 s(\boldsymbol{i}^0)^{\ell-1}s(i_1,i_2^0,\ldots,i_\ell^0),\\
 s_t(i_t)&:=
 s(i_1^0,\ldots,i_{t-1}^0,i_t,i_{t+1}^0,\ldots,i_\ell^0)
 \quad (t\geq2).
 \end{split}
\]
\end{proof}

For $x\in A^\times$, define its sign array by
\[
 s_x((H_1),\ldots,(H_\ell))
 :=\varphi_{((H_1),\ldots,(H_\ell))}(x).
\]
Fix
\[
 \boldsymbol{H}^0=(H_1^0,\ldots,H_\ell^0)
 \in\prod_{i=1}^{\ell}\D_i^c.
\]
For each $t$, define the ring homomorphism
\[
 \begin{split}
 \operatorname{ev}^{(t)}_{\boldsymbol{H}^0}
 &:=
 \varphi^{G_1}_{H_1^0}\otimes\cdots\otimes
 \varphi^{G_{t-1}}_{H_{t-1}^0}\otimes
 {\rm id}_{A_t}\otimes
 \varphi^{G_{t+1}}_{H_{t+1}^0}\otimes\cdots\otimes
 \varphi^{G_\ell}_{H_\ell^0}\\
 &\colon A\longrightarrow A_t.
 \end{split}
\]

\begin{prop}\label{unit-tensor-surjectivity-criterion}
The following conditions are equivalent:
\begin{itemize}
 \item[$(1)$] The homomorphism
 \[
  \mu_\ell\colon\prod_{i=1}^{\ell}A_i^\times\longrightarrow A^\times
 \]
 is surjective.

 \item[$(2)$] For every $x\in A^\times$, the sign array $s_x$ satisfies all relations $(R)$ in Lemma
 $\ref{rectangle-sign-lemma}$.

 \item[$(3)$] For every $x\in A^\times$, put
 \[
  x_t:=\operatorname{ev}^{(t)}_{\boldsymbol{H}^0}(x)\in A_t^\times,
  \qquad
  \epsilon_x:=s_x(\boldsymbol{H}^0)\in\{1,-1\}.
 \]
 Then
 \[
  x=\epsilon_x^{\ell-1}
  (x_1\otimes\cdots\otimes x_\ell).
  \tag{3.4}
 \]
\end{itemize}
\end{prop}

\begin{proof}
Suppose that $\mu_\ell$ is surjective and write
$x=u_1\otimes\cdots\otimes u_\ell$. Then
\[
 s_x((H_1),\ldots,(H_\ell))
 =\prod_{i=1}^{\ell}\varphi^{G_i}_{H_i}(u_i),
\]
so $s_x$ satisfies all relations $(R)$. Thus $(1)$ implies
$(2)$.

Suppose that $(2)$ holds. Since
$\operatorname{ev}^{(t)}_{\boldsymbol{H}^0}$ is a unital ring
homomorphism, $x_t$ is a unit of $A_t$. Moreover,
\[
 \varphi^{G_t}_{H_t}(x_t)
 =s_x(H_1^0,\ldots,H_{t-1}^0,H_t,H_{t+1}^0,\ldots,H_\ell^0).
\]
Lemma $\ref{rectangle-sign-lemma}$ therefore gives, for every
$\boldsymbol{H}=((H_1),\ldots,(H_\ell))$,
\[
 \begin{split}
 s_x(\boldsymbol{H})
 &=\epsilon_x^{\ell-1}
   \prod_{t=1}^{\ell}\varphi^{G_t}_{H_t}(x_t)\\
 &=\varphi_{\boldsymbol{H}}
   \left(\epsilon_x^{\ell-1}
   (x_1\otimes\cdots\otimes x_\ell)\right).
 \end{split}
\]
The injectivity of $\Phi_A$ gives (3.4). Hence $(2)$ implies $(3)$.
Finally, the right-hand side of (3.4) lies in the image of
$\mu_\ell$, so $(3)$ implies $(1)$.
\end{proof}

\begin{cor}\label{two-factor-rectangle-criterion}
For $\ell=2$, the homomorphism
\[
 \mu_2\colon A_1^\times\times A_2^\times
 \longrightarrow(A_1\otimes A_2)^\times
\]
is surjective if and only if, for every
$x\in(A_1\otimes A_2)^\times$,
\[
 \begin{split}
 &(\varphi^{G_1}_{H}\otimes\varphi^{G_2}_{K})(x)
  (\varphi^{G_1}_{H'}\otimes\varphi^{G_2}_{K'})(x)\\
 &\qquad=
  (\varphi^{G_1}_{H}\otimes\varphi^{G_2}_{K'})(x)
  (\varphi^{G_1}_{H'}\otimes\varphi^{G_2}_{K})(x)
 \end{split}
 \tag{3.5}
\]
for all $H,H'\in\D_1$ and $K,K'\in\D_2$.
\end{cor}

\begin{cor}\label{メイン定理B}
One always has
\[
 \left|
 \B{\prod_{i=1}^{\ell}G_i,\prod_{i=1}^{\ell}\D_i}^\times
 \right|
 \geq
 \frac{1}{2^{\ell-1}}
 \prod_{i=1}^{\ell}|\B{G_i,\D_i}^\times|.
 \tag{3.6}
\]
Moreover, the following conditions are equivalent:
\begin{itemize}
 \item[$(1)$] $\mu_\ell$ is surjective;
 \item[$(2)$] every unit of $A$ satisfies all relations $(R)$ in Proposition $\ref{unit-tensor-surjectivity-criterion}$;
 \item[$(3)$]
 \[
 \left|
 \B{\prod_{i=1}^{\ell}G_i,\prod_{i=1}^{\ell}\D_i}^\times
 \right|
 =
 \frac{1}{2^{\ell-1}}
 \prod_{i=1}^{\ell}|\B{G_i,\D_i}^\times|.
 \tag{3.7}
 \]
\end{itemize}
If these equivalent conditions hold and
\[
 \B{G_i,\D_i}^\times
 =\braket{-1_{\B{G_i}},u_1^{(i)},\ldots,u_{r_i}^{(i)}}
\]
for every $i$, then
\[
 \begin{split}
 \B{\prod_{i=1}^{\ell}G_i,\prod_{i=1}^{\ell}\D_i}^\times
 =\left\langle
 \{-1_{\B{G}}\}\cup
 \bigcup_{i=1}^{\ell}
 \{f_i(u_1^{(i)}),\ldots,f_i(u_{r_i}^{(i)})\}
 \right\rangle,
 \end{split}
 \tag{3.8}
\]
where $G=\prod_{i=1}^{\ell}G_i$.
\end{cor}

\begin{proof}
By Corollary $\ref{メイン定理A}$, the ring $A$ is isomorphic to
\[
 \B{\prod_{i=1}^{\ell}G_i,\prod_{i=1}^{\ell}\D_i}.
\]
Lemma $\ref{kernel-unit-tensor-map}$ gives
\[
 |\operatorname{Im}\mu_\ell|
 =\frac{1}{2^{\ell-1}}
  \prod_{i=1}^{\ell}|A_i^\times|.
\]
Since $\operatorname{Im}\mu_\ell\leq A^\times$, inequality (3.6)
holds, and equality holds if and only if $\mu_\ell$ is surjective.
The equivalence with the condition $(R)$ follows from Proposition
$\ref{unit-tensor-surjectivity-criterion}$.

Under the isomorphism $\pi_\ell$ of Corollary
$\ref{メイン定理A}$, one has
\[
 \pi_\ell(u_1\otimes\cdots\otimes u_\ell)
 =\prod_{i=1}^{\ell}f_i(u_i).
\]
Therefore surjectivity, together with the given generating sets of
the $A_i^\times$, yields (3.8).
\end{proof}

For completeness, the criterion can also be expressed entirely in
terms of the mark matrices. Choose an ordering of each $\D_i^c$, let
$M_i$ be the corresponding mark matrix of $A_i$, and put
\[
 M:=M_1\otimes\cdots\otimes M_\ell.
\]

\begin{cor}\label{mark-matrix-surjectivity-criterion}
The homomorphism $\mu_\ell$ is surjective if and only if every sign
array
\[
 \epsilon\in
 \{1,-1\}^{\D_1^c\times\cdots\times\D_\ell^c}
\]
satisfying
\[
 M^{-1}\epsilon\in
 \Z^{|\D_1^c|\cdots|\D_\ell^c|}
 \tag{3.9}
\]
satisfies all relations $(R)$.
\end{cor}

\begin{proof}
With respect to the standard transitive-set bases, the matrix of
$\Phi_A$ is $M$. Hence a sign array $\epsilon$ is the mark array of
an element of $A$ if and only if $M^{-1}\epsilon$ has integral
entries. By Lemma $\ref{tensor-Burnside-homomorphism}$, every such
element is a unit. The assertion now follows from Proposition
$\ref{unit-tensor-surjectivity-criterion}$.
\end{proof}

We record a sufficient condition which will be used for parabolic
partial Burnside rings.

\begin{lem}\label{two-factor-small-unit-surjectivity}
Let
\[
 A=\B{G,\D},\qquad B=\B{K,\mathcal E}.
\]
Assume that
\[
 \operatorname{Idem}(A)=\operatorname{Idem}(B)=\{0,1\}
 \qquad\text{and}\qquad |B^\times|\leq 4.
\]
Then the canonical homomorphism
\[
 A^\times\times B^\times\longrightarrow(A\otimes B)^\times,
 \qquad (u,v)\longmapsto u\otimes v,
\]
is surjective. Moreover,
\[
 \operatorname{Idem}(A\otimes B)=\{0,1\}.
\]
\end{lem}

\begin{proof}
Every unit of a partial Burnside ring has square one. Suppose
first that $B^\times=\{1,-1\}$. For $z\in(A\otimes B)^\times$ and
$H\in\D^c$, the element
\[
 (\varphi_H^G\otimes\operatorname{id}_B)(z)
\]
is either $1$ or $-1$. Hence the $H$-th row of the sign matrix of
$z$ is constant. All relations $(R)$ therefore hold, and
Corollary $\ref{two-factor-rectangle-criterion}$ gives
surjectivity.

Suppose that $|B^\times|=4$, and choose
$v\in B^\times\setminus\{1,-1\}$. Then
\[
 B^\times=\braket{-1,v}.
\]
Write
\[
 \varphi^K_L(v)=(-1)^{q_L},\qquad q_L\in\mathbb F_2=\{0,1\},
\]
where we identify $\mathbb F_2$ with $\{0,1\}$.
The function $L\mapsto q_L$ is not constant, since otherwise the
injectivity of the Burnside homomorphism would give $v=\pm1$.

Let $z\in(A\otimes B)^\times$. For every $H\in\D^c$, the element
\[
 (\varphi^G_H\otimes\operatorname{id}_B)(z)
\]
is a unit of $B$. Hence there are $a_H,b_H\in\mathbb F_2$ such that
\[
 (\varphi^G_H\otimes\varphi^K_L)(z)
 =(-1)^{a_H+b_H q_L}
 \tag{3.10}
\]
for every $L\in\mathcal E^c$. Choose $L_0,L_1\in\mathcal E^c$ with
$q_{L_0}=0$ and $q_{L_1}=1$, and put
\[
 a=(\operatorname{id}_A\otimes\varphi^K_{L_0})(z),\qquad
 a'=(\operatorname{id}_A\otimes\varphi^K_{L_1})(z),\qquad
 u=aa'.
\]
Then $a,a',u\in A^\times$ and
\[
 \varphi^G_H(a)=(-1)^{a_H},\qquad
 \varphi^G_H(u)=(-1)^{b_H}.
\]
After multiplying $z$ by $a\otimes1$, we may therefore assume that
its marks are $(-1)^{b_H q_L}$.

In $(A\otimes B)\otimes\mathbb Q$, the element with these marks is
\[
 1-\frac{(1-u)\otimes(1-v)}{2}.
 \tag{3.11}
\]
Indeed, both sides have the same marks, and the rationalized Burnside
homomorphism is injective. Since the element in (3.11) belongs to
$A\otimes B$, one has
\[
 (1-u)\otimes(1-v)\in2(A\otimes B).
\]
Since $A$ and $B$ are free abelian, reduction modulo $2$
identifies
\[
 (A\otimes_{\Z}B)/2(A\otimes_{\Z}B)
 \simeq (A/2A)\otimes_{\mathbb F_2}(B/2B).
\]
Thus reducing modulo $2$ gives
\[
 \overline{1-u}\otimes\overline{1-v}=0
 \quad\text{in}\quad
 (A/2A)\otimes_{\mathbb F_2}(B/2B).
\]
A pure tensor over a field is zero only if one of its factors is
zero. Thus either $1-u\in2A$ or $1-v\in2B$.

In the first case, $(1-u)/2$ is an idempotent of $A$, because
$u^2=1$ and $A$ is torsion-free. Hence $u=\pm1$. Similarly, the
second case gives $v=\pm1$. In either case (3.11) is a pure tensor:
it is one of
\[
 1\otimes1,\qquad 1\otimes v,\qquad u\otimes1.
\]
Consequently $z$ is a pure tensor of units, and the canonical
homomorphism is surjective.

It remains to prove the assertion about idempotents. Let
$e\in A\otimes B$ satisfy $e^2=e$. Then $w=1-2e$ is a unit and
$w\equiv1\pmod{2(A\otimes B)}$. By surjectivity, write $w=u\otimes v$
with $u\in A^\times$ and $v\in B^\times$. Modulo $2$,
\[
 \bar u\otimes\bar v=\bar1\otimes\bar1.
\]
The reductions $\bar u$ and $\bar v$ are nonzero because they
are units. Equality of nonzero pure tensors over a field implies that
the corresponding factors differ by a nonzero scalar. Since
$\mathbb F_2^\times=\{1\}$, it follows that
$u\equiv1\pmod{2A}$ and $v\equiv1\pmod{2B}$. Therefore
$(1-u)/2$ and $(1-v)/2$ are idempotents. By hypothesis,
$u,v\in\{1,-1\}$, so $w=\pm1$ and hence $e=0$ or $1$.
\end{proof}

\begin{cor}\label{small-unit-surjectivity}
For $i=1,\ldots,\ell$, let $A_i=\B{G_i,\D_i}$. Assume that
\[
 \operatorname{Idem}(A_i)=\{0,1\}
 \qquad\text{and}\qquad
 |A_i^\times|\leq4
\]
for every $i$. Then
\[
 \mu_\ell\colon\prod_{i=1}^{\ell}A_i^\times
 \longrightarrow
 \left(\bigotimes_{i=1}^{\ell}A_i\right)^\times
\]
is surjective. Moreover,
\[
 \operatorname{Idem}\left(\bigotimes_{i=1}^{\ell}A_i\right)
 =\{0,1\}.
\]
\end{cor}

\begin{proof}
Set $C_1=A_1$ and, for $j\geq2$, set
\[
 C_j:=A_1\otimes\cdots\otimes A_j.
\]
By Corollary $\ref{メイン定理A}$, each $C_j$ is isomorphic to a
partial Burnside ring. Assume inductively that
$\operatorname{Idem}(C_{j-1})=\{0,1\}$ and that the canonical map
\[
 \prod_{i=1}^{j-1}A_i^\times\longrightarrow C_{j-1}^\times
\]
is surjective. Applying Lemma
$\ref{two-factor-small-unit-surjectivity}$ to $C_{j-1}$ and $A_j$
shows that
\[
 C_{j-1}^\times\times A_j^\times\longrightarrow C_j^\times
\]
is surjective and that $\operatorname{Idem}(C_j)=\{0,1\}$. Composing
the two surjections completes the induction.
\end{proof}

\section{PBR of a reducible finite Coxeter group}
\begin{dfn}
	Let $W$ be a finite Coxeter group with Coxeter system $(W,S)$. Then a subgroup $P$ of $W$ is called a parabolic subgroup if there exists $J\subseteq S$ such that $(P)=(\braket{J})$. Denote by $\P_W$ the set of all parabolic subgroups of $W$.
\end{dfn}
Let $W$ be a finite Coxeter group. Then the set $\P_W$ of all parabolic subgroups of $W$ is a collection (see \cite{so76}). Thus $\B{W,\P_W}$ is the PBR of $W$ relative to $\P_W$. 
For the PBR of a finite Coxeter group relative to the set of all parabolic subgroups, the following lemma is well known.
\begin{lem}[\cite{so66}]\label{Signunit}
	Let $W$ be a finite Coxeter group with Coxeter system $(W,S)$. Then the element
	\[
	\varepsilon_W:=\sum_{J\subseteq S}(-1)^{|J|}[W/\braket{J}]
	\] 
	lies in $\B{W,\P_W}^\times$. Moreover, for $P\in\P_W$ with $(P)=(\braket{J})$,
	\[
	\varphi^W_P(\varepsilon_W)=(-1)^{|J|}.
	\]
\end{lem}
We call $\varepsilon_W$ a \emph{sign unit} of $\B{W,\P_W}$. 

\begin{thm}\label{theorem}
Let $(W,S)$ be a reducible finite Coxeter
system with irreducible components $(W_i,S_i)$ for
$i=1,\ldots,\ell$, so that
\[
 W=\prod_{i=1}^{\ell}W_i,
 \qquad
 S=\bigsqcup_{i=1}^{\ell}S_i.
\]
Then the following hold:
\begin{itemize}
 \item[$(1)$] The canonical homomorphism
 \[
  \prod_{i=1}^{\ell}\B{W_i,\P_{W_i}}^\times
  \longrightarrow
  \B{W,\P_W}^\times
 \]
 is surjective if and only if
 \[
  |\B{W,\P_W}^\times|
  =\frac{1}{2^{\ell-1}}
   \prod_{i=1}^{\ell}|\B{W_i,\P_{W_i}}^\times|.
 \]
 Equivalently, every unit satisfies the relations $(R)$ of
 Proposition $\ref{unit-tensor-surjectivity-criterion}$.

 \item[$(2)$] If $\varepsilon_W$ is the sign unit of
 $\B{W,\P_W}$, then
 \[
  \varepsilon_W=\prod_{i=1}^{\ell}f_i(\varepsilon_{W_i}).
 \]
\end{itemize}
\end{thm}

\begin{proof}
Every subset $J\subseteq S$ decomposes uniquely as
$J=\bigsqcup_{i=1}^{\ell}J_i$ with $J_i\subseteq S_i$, and
$\langle J\rangle=\prod_i\langle J_i\rangle$. Conjugation in
$W=\prod_{i=1}^{\ell}W_i$ is componentwise. Hence
\[
 \P_W=\prod_{i=1}^{\ell}\P_{W_i}.
\]
Assertion $(1)$ now follows from Corollary $\ref{メイン定理B}$.

Let $J\subseteq S$ and write
$J=\bigsqcup_{i=1}^{\ell}J_i$ with $J_i\subseteq S_i$. Then
\[
 \braket{J}=\prod_{i=1}^{\ell}\braket{J_i}.
\]
For $P\in\P_W$ with $(P)=(\braket{J})$, marks depend only on
the conjugacy class of the subgroup. Hence Lemma
$\ref{バーンサイド準同型の分解補題}$ gives
\[
 \begin{split}
 \varphi^W_P\left(\prod_{i=1}^{\ell}f_i(\varepsilon_{W_i})\right)
 &=\prod_{i=1}^{\ell}
   \varphi^{W_i}_{\braket{J_i}}(\varepsilon_{W_i})\\
 &=(-1)^{\sum_{i=1}^{\ell}|J_i|}
 =(-1)^{|J|}
 =\varphi^W_P(\varepsilon_W).
 \end{split}
\]
The injectivity of the Burnside homomorphism $\varphi^W$ proves
$(2)$.
\end{proof}

The following results are known for irreducible finite Coxeter groups.
\begin{thm}[\cite{io15}]
Let $W$ be an irreducible finite Coxeter group of type {\rm A}. Then
\[
|\B{W,\P_W}^\times| = 4.
\]
\end{thm}

\begin{thm}[\cite{oty17}]
	Let $W$ be an irreducible finite Coxeter group. If $W$ is of type
	${\rm B}, {\rm D}, {\rm E}_6, {\rm E}_7$, or ${\rm E}_8$, then
	\[
	|\B{W,\P_W}^\times| = 4.
	\]
\end{thm}
The remaining rank-two case is also elementary.
\begin{thm}
Let $W=W({\rm I}_2(m))$ be a Coxeter group of type ${\rm I}_2(m)$,
where $m\geq3$. Then
\[
 |\B{W,\P_W}^\times|=4.
\]
\end{thm}

\begin{proof}
Write
\[
 W=\langle s,t\mid s^2=t^2=(st)^m=1\rangle.
\]
If $m$ is odd, the two rank-one standard parabolic subgroups are
conjugate. Ordering the parabolic conjugacy classes by
\[
 W,\quad\langle s\rangle,\quad1
\]
and using the corresponding transitive-set basis, the mark matrix is
\[
 M_{\rm odd}=
 \begin{pmatrix}
  1&0&0\\
  1&1&0\\
  1&m&2m
 \end{pmatrix}.
\]
A sign vector $\epsilon\in\{1,-1\}^3$ lies in the image of the
Burnside homomorphism if and only if
$M_{\rm odd}^{-1}\epsilon$ is integral. A direct calculation shows
that this holds precisely for
\[
 (1,1,1),\quad(1,-1,1),\quad
 (-1,1,-1),\quad(-1,-1,-1).
\]

If $m$ is even, the subgroups $\langle s\rangle$ and
$\langle t\rangle$ represent the two conjugacy classes of reflections.
With the order
\[
 W,\quad\langle s\rangle,\quad\langle t\rangle,\quad1,
\]
the mark matrix is
\[
 M_{\rm even}=
 \begin{pmatrix}
  1&0&0&0\\
  1&2&0&0\\
  1&0&2&0\\
  1&m&m&2m
 \end{pmatrix}.
\]
The integral sign vectors are precisely
\[
 (1,1,1,1),\quad(1,-1,-1,1),\quad
 (-1,1,1,-1),\quad(-1,-1,-1,-1).
\]
In either case there are exactly four sign vectors in the image.
The unit criterion for the Burnside homomorphism therefore gives
$|\B{W,\P_W}^\times|=4$.
\end{proof}
Theorem $\ref{theorem}$, together with Corollary
$\ref{small-unit-surjectivity}$, gives the following.
\begin{cor}\label{coxeter-order-four-corollary}
Let $(W,S)$ be a reducible finite Coxeter
system with irreducible components $(W_t,S_t)$ for
$t=1,\ldots,\ell$. If
\[
 |\B{W_t,\P_{W_t}}^\times|=4
\]
for every $t$, then the following hold:
\begin{itemize}
 \item[$({\rm i})$]
 \[
  |\B{W,\P_W}^\times|=2^{\ell+1}.
 \]
 \item[$({\rm ii})$]
 \[
  \B{W,\P_W}^\times
  =\braket{-1_{\B{W}},f_1(\varepsilon_{W_1}),\ldots,
  f_\ell(\varepsilon_{W_\ell})},
 \]
 where $\varepsilon_{W_i}$ is the sign unit of
 $\B{W_i,\P_{W_i}}$.
\end{itemize}
\end{cor}

\begin{proof}
	Put
	\[
	A_t:=\B{W_t,\P_{W_t}}
	\qquad (t=1,\ldots,\ell),
	\]
	and write
	\[
	\varepsilon_t:=\varepsilon_{W_t}.
	\]
	
	We first show that
	\[
	\operatorname{Idem}(A_t)=\{0,1\}
	\]
	for every $t$. By Lemma~\ref{Signunit}, $\varepsilon_t$ is a unit of $A_t$
	distinct from $\pm1_{A_t}$. Since
	\[
	|A_t^\times|=4,
	\]
	we have
	\[
	A_t^\times
	=
	\{1_{A_t},-1_{A_t},\varepsilon_t,-\varepsilon_t\}
	=
	\langle-1_{A_t},\varepsilon_t\rangle.
	\]
	
	Let $e\in A_t$ be an idempotent. Then
	\[
	u:=1_{A_t}-2e
	\]
	is a unit of $A_t$, since
	\[
	u^2=(1-2e)^2=1_{A_t}.
	\]
	Moreover,
	\[
	u-1_{A_t}\in 2A_t.
	\]
	Since the coefficient of $[W_t/1]$ in $\varepsilon_t$ is
	$1$, we have
	\[
	\pm\varepsilon_t-1_{A_t}\notin 2A_t.
	\]
	Therefore,
	\[
	u\in\{1_{A_t},-1_{A_t}\}.
	\]
	As $A_t$ is torsion-free, $u=1_{A_t}$ gives $e=0$, while $u=-1_{A_t}$ gives
	$e=1$. Therefore
	\[
	\operatorname{Idem}(A_t)=\{0,1\}.
	\]
	
	Corollary~\ref{small-unit-surjectivity} now shows that the canonical
	homomorphism
	\[
	\prod_{t=1}^{\ell}A_t^\times
	\longrightarrow
	\B{W,\P_W}^\times
	\]
	is surjective. Hence Corollary~\ref{メイン定理B} gives
	\[
	|\B{W,\P_W}^\times|
	=
	\frac{1}{2^{\ell-1}}
	\prod_{t=1}^{\ell}|A_t^\times|
	=
	\frac{4^\ell}{2^{\ell-1}}
	=
	2^{\ell+1}.
	\]
	This proves ${\rm (i)}$.
	
	Finally,
	\[
	A_t^\times=\langle-1_{A_t},\varepsilon_t\rangle
	\]
	for every $t$. Therefore, by Corollary~\ref{メイン定理B},
	\[
	\B{W,\P_W}^\times
	=
	\left\langle
	-1_{\B{W}},
	f_1(\varepsilon_{W_1}),\ldots,
	f_\ell(\varepsilon_{W_\ell})
	\right\rangle.
	\]
	This proves ${\rm (ii)}$.
\end{proof}

\end{document}